\documentclass[preprint,12pt]{elsarticle}

\usepackage{graphicx}%
\usepackage{multirow}%

\usepackage{amsthm}%
\usepackage{mathrsfs}%
\usepackage{xcolor}%
\usepackage{textcomp}%
\usepackage{manyfoot}%
\usepackage{booktabs}%
\usepackage{listings}%

\usepackage{mathtools}
\usepackage{amsmath, amssymb}%
\newtheorem{theorem}{Theorem}
\newtheorem{lemma}[theorem]{Lemma} % lemma 与 theorem 共享编号
\allowdisplaybreaks

\usepackage[mathscr]{eucal}

\usepackage{mismath}

\usepackage{hyperref}

\usepackage{esint}

\graphicspath{{figs/}}
\usepackage{pgf}
\usepackage{pgfplots}
\pgfplotsset{compat=1.14}

\usepackage{tikz}
\usetikzlibrary{arrows.meta}
\tikzset{%
>={Latex[width=2mm,length=2mm]},
base/.style = {rectangle, rounded corners, draw=black, minimum width=2cm, minimum height=0.75cm, text centered, font=\sffamily}, activityStarts/.style = {base, fill=blue!30}, startstop/.style = {base, fill=red!30}, activityRuns/.style = {base, fill=green!30}, process/.style = {base, minimum width=2.5cm, fill=orange!15}, }

\usetikzlibrary{math}
\usetikzlibrary{calc}

\usepackage{subcaption}

\usepackage{makecell}
\setcellgapes{2pt}

\usepackage{booktabs}

\usepackage{siunitx}
\usepackage{algorithm}
\usepackage{algpseudocode}
\algrenewcommand\algorithmiccomment[2][\normalsize]{{#1\hfill #2}}

\usepackage[capitalise]{cleveref}

\usepackage{easyReview}
\usepackage{xcolor}
\definecolor{MyColor1}{HTML}{BFCCB5}
\definecolor{MyColor2}{HTML}{7C96AB}
\definecolor{MyColor3}{HTML}{B7B7B7}
\definecolor{MyColor4}{HTML}{EDC6B1}

\makeatletter
\def\ps@pprintTitle{%
  \let\@oddhead\@empty
  \let\@evenhead\@empty
  \let\@oddfoot\@empty
  \let\@evenfoot\@empty
}
\makeatother

\begin{document}
  \begin{frontmatter}
    \title{An algebraic multiscale preconditioner for large sparse SPD matrices}


    \author[CHUK]{Yingjie Zhou}

    \author[EIAS]{Shubin Fu\corref{cor1}}
    \cortext[cor1]{Corresponding author.} \ead{sfu@eitech.edu.cn}

    \author[CUHK]{Eric T. Chung}



    %% \author{}

    \affiliation[CUHK]{ organization={Department of Mathematics, The Chinese University of Hong Kong}, city={Shatin}, country={Hong Kong SAR} }

    \affiliation[EIAS]{ organization={Eastern Institute for Advanced Study},
    %%    addressline={},
    city={Ningbo}, postcode={315200}, state={Zhejiang}, country={PR China} }

    \begin{abstract}
          We present a two-grid algebraic multiscale preconditioner for large sparse symmetric
    positive definite systems arising from elliptic problems with highly heterogeneous
    coefficients. The coarse space is constructed directly from the system
    matrix by graph partitioning and local generalized eigenvalue solvers, yielding
    basis functions that capture the low-energy modes responsible for slow convergence.
    The method requires no geometric information, making it suitable for unstructured
    and matrix-only settings, and its construction is naturally parallelizable. Numerical
    results for heterogeneous Darcy flow problems show robustness with respect to
    coefficient contrast and problem size, better performance than standard
    algebraic multigrid on challenging large-scale cases, and good parallel scalability.
    \end{abstract}

    %% Graphical abstract
    %% \begin{graphicalabstract}
    %% \includegraphics{grabs}
    %% \end{graphicalabstract}

    %% Research highlights
    %% \begin{highlights}
    %% \item Research highlight 1
    %% \item Research highlight 2
    %% \end{highlights}

  \end{frontmatter}

  \section{Introduction}
  \label{sec:introduction}

  In this paper, we consider the linear system of equations
  \begin{equation}
    \label{eq:matrix-form}\mathsf{A}\mathsf{u}= \mathsf{f},
  \end{equation}
  where $\mathsf{A}\in \mathbb{R}^{n \times n}$ is the system matrix, $\mathsf{f}$
  is the source term, and $\mathsf{u}$ is the solution vector to be determined. A
  significant challenge arises from the complex geological structures of the subsurface,
  which introduce highly heterogeneous and anisotropic coefficients into the
  PDEs. This often results in a severely ill-conditioned matrix $\mathsf{A}$,
  posing a substantial hurdle for numerical solvers.

  For large-scale problems, direct solvers are often computationally infeasible due
  to their prohibitive memory and computational demands. Consequently, iterative
  methods based on Krylov subspaces, such as the Conjugate Gradient (CG) or
  GMRES methods, are the preferred choice. However, the convergence rate of these
  solvers is highly dependent on the condition number of the matrix $\mathsf{A}$.
  Without an effective preconditioner, their performance can be unacceptably slow,
  particularly for problems with high-contrast material properties.

  Among the vast landscape of preconditioning techniques, multigrid methods are distinguished
  by their potential for optimal or near-optimal computational complexity. These
  methods accelerate convergence by addressing error components across multiple
  scales, smoothing high-frequency errors on fine grids and correcting low-frequency
  errors on coarser grids. Traditional geometric multigrid (GMG) methods, however,
  require a hierarchy of well-defined geometric grids, which can be difficult to
  construct for problems with complex geometries and may not effectively capture
  the multiscale nature of the underlying physics in heterogeneous media.

  Algebraic multigrid (AMG) methods offer a more flexible alternative by constructing
  coarse levels based solely on the algebraic properties of the matrix $\mathsf{A}$,
  eliminating the need for geometric information. While powerful, standard AMG
  approaches can still falter when faced with strong anisotropy or large
  discontinuities in coefficients, as their coarsening strategies, based on
  algebraic connectivity, may not align with the physical characteristics of the
  problem.

  To address these challenges, multiscale methods have emerged as a powerful
  paradigm. Techniques such as the multiscale finite element method (MsFEM) \cite{hou1997multiscale},
  the generalized multiscale finite element method (GMsFEM)
  \cite{chung2023multiscale, efendiev2013generalized}, and the multiscale finite
  volume method \cite{jenny2003multi, hajibeygi2009multiscale} have proven effective.
  Adapting these techniques to construct preconditioners is a promising strategy
  for accelerating iterative solvers
  \cite{fu2024adaptive, ye2024robust, ye2025highly, zhou2024robust}. Furthermore,
  these multiscale concepts can be formulated in a purely algebraic manner
  \cite{heinlein2025algebraic, wang2014algebraic}, making them broadly applicable
  to problems on unstructured meshes or scenarios where only the system matrix is
  available.

  In this work, we build upon these principles to develop a novel two-level
  algebraic multiscale preconditioner. The core innovation of our method is the construction
  of a problem-dependent coarse space through purely algebraic means.
  Specifically, we partition the graph of the matrix $\mathsf{A}$ to define local
  subdomains and then construct multiscale basis functions by solving generalized
  eigenvalue problems within each subdomain. These basis functions are designed
  to capture the low-energy modes responsible for slow convergence. By incorporating
  these functions into a two-level framework, our preconditioner effectively
  handles the complex, multiscale interactions induced by heterogeneous coefficients
  without requiring any geometric information.

  The main contributions of this paper are threefold: (1) the development of a
  fully algebraic and parallelizable framework for constructing multiscale
  coarse spaces; (2) the systematic integration of local spectral information to
  create robust basis functions for high-contrast Darcy flow problems; and (3)
  extensive numerical validation demonstrating the robustness and scalability of
  the proposed method, showing superior performance compared to standard AMG for
  challenging large-scale problems.

  The remainder of this paper is organized as follows. Section~\ref{sec:model-problem}
  provides an overview of the two-level preconditioning framework and the model
  problem. Section~\ref{sec:methods} details the construction of our algebraic
  multiscale preconditioner, including the domain partitioning and the
  formulation of local eigenvalue problems. In Section~\ref{sec:numerical-experiments},
  we present a series of numerical experiments to evaluate the performance and
  robustness of the proposed method. Finally, Section~\ref{sec:conclusion}
  offers concluding remarks and discusses potential directions for future
  research.

  \section{Problem setting}
  \label{sec:model-problem} Systems from \Cref{eq:matrix-form} frequently arise
  in computational science and engineering applications, particularly in the numerical
  simulation of physical phenomena governed by elliptic partial differential equations.

  More specificially, these linear systems typically originate from the
  discretization of variational problems that model diverse physical processes
  including groundwater flow, oil reservoir simulation, heat conduction, structural
  mechanics, and electromagnetic field computations. The mathematical foundation
  of these problems can be formulated as follows: Find $p \in V$ such that
  \begin{equation}
    \label{eq:variational}a_{\Omega}(p, v) = l(v) \quad \forall v \in V,
  \end{equation}
  where $V$ is a suitable Hilbert space (typically $H^{1}(\Omega)$ or a subspace
  thereof), $a(\cdot, \cdot): V \times V \to \mathbb{R}$ is a symmetric,
  continuous, and coercive bilinear form, $l(\cdot): V \to \mathbb{R}$ is a
  continuous linear functional, and $\Omega \subset \mathbb{R}^{d}$ ($d=2,3$) represents
  the computational domain with appropriate boundary conditions.

  A prototypical example of such variational problems, which is central to our
  study, is the steady-state diffusion equation with heterogeneous coefficients:
  \begin{equation}
    a_{\Omega}(p, v) = \int_{\Omega}\kappa(x) \nabla p \cdot \nabla v \, d\mathsf{x}
    , \quad \text{and}\quad l(v) = \int_{\Omega}f(x) v \, d\mathsf{x},
  \end{equation}
  where $\kappa(x)$ is a spatially varying diffusion tensor that is symmetric
  and uniformly positive definite, and $f(x)$ represents distributed source terms.
  In the context of subsurface flow applications, $p$ typically represents the pressure
  field, $\kappa$ corresponds to the permeability tensor characterizing the
  porous medium, and $f$ represents sources and sinks.

  Upon discretization using finite element, finite difference, or finite volume methods,
  the variational problem \eqref{eq:variational} leads to the linear system \eqref{eq:matrix-form}.
  A fundamental challenge in solving \eqref{eq:matrix-form} arises when the
  coefficient $\kappa(x)$ exhibits strong heterogeneity, leading to severe ill-conditioning
  of the matrix $\mathsf{A}$. This situation is particularly prevalent in
  applications involving high-contrast media arising from porous media flow and
  composite materials, where the coefficient values can span several orders of magnitude
  across the computational domain.
  % In such scenarios, the condition number $\kappa(\mathsf{A}) = \lambda_{\max}(\mathsf{A})/\lambda_{\min}(\mathsf{A})$ can become extremely large, often scaling with the contrast ratio of the coefficients. This ill-conditioning poses significant computational challenges, as standard iterative solvers such as the conjugate gradient (CG) method exhibit convergence rates that deteriorate with increasing condition number. Specifically, the convergence rate of CG is bounded by $(\sqrt{\kappa(\mathsf{A})}- 1)/(\sqrt{\kappa(\mathsf{A})}+ 1)$, which approaches unity as $\kappa(\mathsf{A})$ increases, leading to prohibitively slow convergence or even practical stagnation.

  \section{Methods}
  \label{sec:methods} In this section, we present a two-level algebraic
  multiscale preconditioner designed to efficiently solve the large, sparse, and
  ill-conditioned linear system \eqref{eq:matrix-form}. Our approach is rooted
  in the principles of domain decomposition and multiscale methods, aiming to construct
  a robust solver that is effective even for challenging scenarios, such as
  those involving high-contrast media. The fundamental idea is to build a problem-dependent
  coarse space that accurately captures the low-frequency (or algebraically
  smooth) error components that are responsible for the slow convergence of standard
  iterative methods. Unlike traditional geometric multigrid methods, which rely on
  a hierarchy of geometrically coarsened grids, our coarse space is constructed
  algebraically. This is achieved by solving local generalized eigenvalue problems
  on a set of overlapping subdomains that partition the computational domain.
  The eigenvectors corresponding to the smallest eigenvalues of these local
  problems represent the low-energy modes of the system and are used as basis
  functions for the coarse space. By incorporating these multiscale basis functions,
  the resulting preconditioner can effectively handle the complex interactions
  across different scales induced by the heterogeneous coefficients, leading to a
  substantial improvement in solver performance.

  \subsection{Two-level algebraic multiscale preconditioner}
  The construction of our two-level preconditioner commences with a decomposition
  of the computational domain $\Omega$ into a set of non-overlapping coarse subdomains,
  denoted by $\{\Omega_{i}\}_{i=1}^{N}$, where $N$ represents the total number of
  these subdomains. A key feature of our method is that this partitioning is
  performed in a purely algebraic manner, directly leveraging the information encoded
  in the system matrix $\mathsf{A}$. Specifically, we operate on the graph associated
  with $\mathsf{A}$, where the nodes correspond to the degrees of freedom and
  the edges represent the couplings between them. This algebraic approach obviates
  the need for any explicit geometric information or a hierarchy of grids,
  making the method particularly well-suited for problems defined on complex, unstructured
  meshes or when only the matrix itself is available. The resulting subdomains
  are therefore determined by the connectivity and strength of connections
  within the matrix, ensuring that the partitioning is adapted to the underlying
  physical problem.

  \begin{figure*}[!ht]
    \centering
    \begin{subfigure}
      [b]{0.4\textwidth}
      \centering
      \includegraphics[width=\textwidth]{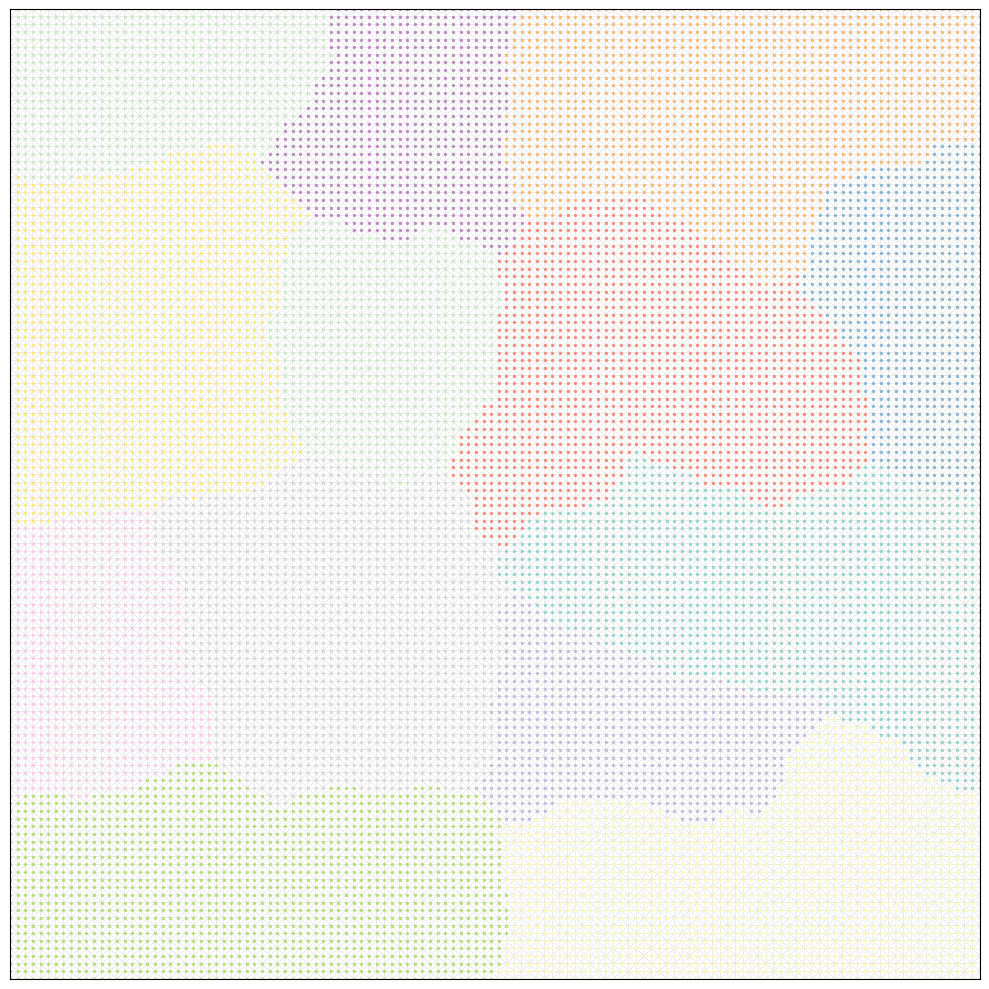}
      \caption{2D graph partition for $128^{2}$ grid}
    \end{subfigure}
    \hfill
    \begin{subfigure}
      [b]{0.45\textwidth}
      \centering
      \includegraphics[width=\textwidth]{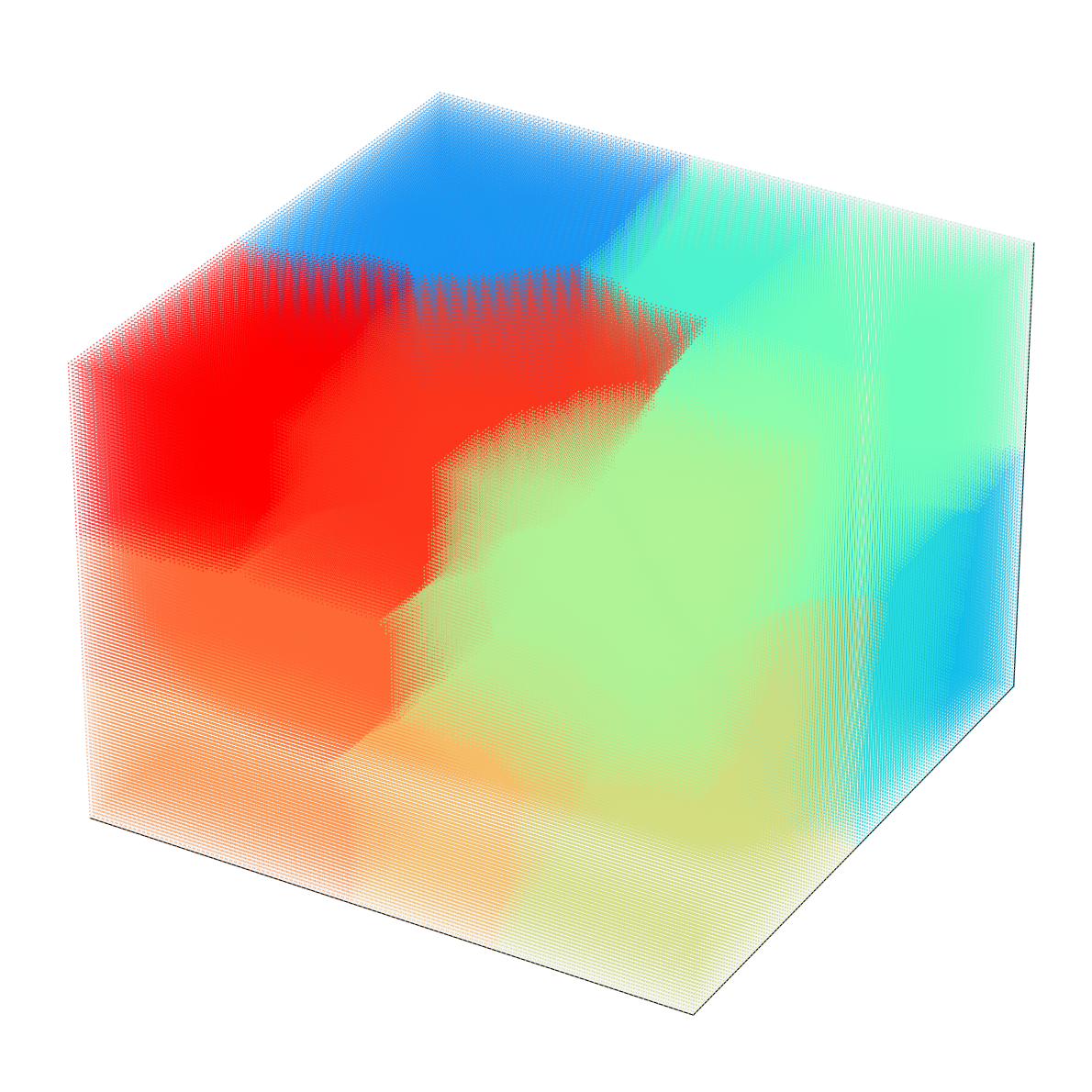}
      \caption{3D graph partition for $128^{3}$ grid}
    \end{subfigure}
    \caption{Graph partition examples using METIS to get 64 subdomains of the
    grid.}
    \label{fig:domain-decomposition}
  \end{figure*}

  Following the domain decomposition, we construct the coarse space, which is
  fundamental to the multiscale nature of our preconditioner. This space is spanned
  by a set of locally supported basis functions, $\{\Phi_{i}\}_{i=1}^{N_c}$, where
  $N_{c}$ is the dimension of the coarse space. Each basis function $\Phi_{i}$
  is associated with a coarse subdomain $\Omega_{i}$ and has a compact support, typically
  confined to an overlapping region constructed around $\Omega_{i}$. The
  resulting coarse space is formally defined as the span of these basis
  functions:
  \[
    W_{H}^{\text{c}}\coloneqq\text{span}\{\Phi_{i}\}_{i=1}^{N_c}.
  \]
  The relationship between the coarse space and the original fine-scale space is
  established through a prolongation operator, denoted by
  $\mathsf{R}_{0}^{\intercal}$. The columns of this operator are the vector representations
  of the coarse basis functions $\{\Phi_{i}\}$ in the fine-scale discretization.
  This operator maps a vector from the coarse space to its equivalent
  representation in the fine-scale space. Subsequently, the coarse-scale system
  matrix $\mathsf{A}_{0}$ is formulated via the Galerkin projection of the fine-scale
  operator $\mathsf{A}$ onto the coarse space: $\mathsf{A}_{0}\coloneqq \mathsf{R}
  _{0}\mathsf{A}\mathsf{R}_{0}^{\intercal}$, where $\mathsf{R}_{0}$ is the
  corresponding restriction operator (the transpose of $\mathsf{R}_{0}^{\intercal}$).
  The efficacy of the entire two-level scheme hinges critically on the properties
  of this coarse space. An ideal coarse space should be of low dimensionality to
  ensure that the coarse problem remains computationally inexpensive, yet it must
  be rich enough to accurately approximate the low-energy modes of the system that
  are poorly attenuated by standard relaxation smoothers. Therefore, the design of
  a compact and effective coarse space is the most pivotal aspect of developing a
  robust and efficient two-level preconditioner.

  The fine-scale linear system \eqref{eq:matrix-form} will be solved by
  preconditioned iterative solvers with a two-level multiscale preconditioner of
  the form:
  \[
    \mathsf{P}^{-1}= \mathsf{R}_{0}^{\intercal}\mathsf{A}_{0}^{\dagger}\mathsf{R}
    _{0}.
  \]
  where $\mathsf{A}_{0}^{\dagger}$ is the pseudo-inverse of $\mathsf{A}_{0}$.

  \subsection{Multiscale coarse spaces}
  The cornerstone of our proposed method lies in the systematic construction of a
  multiscale coarse space, $W_{H}^{\text{c}}$, which is meticulously designed to
  capture the low-energy modes of the system matrix $\mathsf{A}$. These modes, often
  referred to as algebraically smooth components, are responsible for the slow
  convergence of standard iterative solvers. To this end, we adopt and adapt the
  methodology of the Generalized Multiscale Finite Element Method (GMsFEM), which
  excels at generating problem-dependent basis functions by solving local spectral
  problems.

  In the standard GMsFEM, these basis functions are derived from a local
  generalized eigenvalue problem defined on each subdomain $\Omega_{i}$: find $\psi
  _{k}^{(i)}\in V(\Omega_{i})$ and $\lambda_{k}^{(i)}\in \mathbb{R}$ such that
  \begin{equation*}
    \label{eq:eigenproblem}a_{i}(\psi_{k}^{(i)}, v) = \lambda_{k}^{(i)}s_{i}(\psi
    _{k}^{(i)}, v), \quad \forall v \in V(\Omega_{i}),
  \end{equation*}

  %这里加上我们实际求解的不是这个方程
  where $V(\Omega_{i})$ is a local function space, and the bilinear forms $a_{i}(
  \cdot, \cdot)$ and $s_{i}(\cdot, \cdot)$ are formulated based on the underlying
  partial differential equation and boundary conditions.

  Our approach translates this concept into a purely algebraic framework,
  obviating the need for geometric information. The local stiffness matrix $\mathsf{A}
  _{i}$ for each subdomain $\Omega_{i}$ is obtained by restricting the global matrix
  $\mathsf{A}$ to the degrees of freedom within that subdomain, i.e.,
  $\mathsf{A}_{i}\coloneqq \mathsf{R}_{i}\mathsf{A}\mathsf{R}_{i}^{\intercal}$, where
  $\mathsf{R}_{i}$ is the corresponding Boolean restriction operator. To ensure
  the resulting local problem is well-posed and physically meaningful, we construct
  a modified matrix, $\widehat{\mathsf{A}_i}$, from $\mathsf{A}_{i}$. This is
  achieved by adjusting the diagonal entries so that each diagonal element
  becomes the negative sum of the off-diagonal elements in its row. This procedure
  enforces that the row sums of $\widehat{\mathsf{A}_i}$ are zero, making it a
  symmetric positive semi-definite matrix. This is algebraically analogous to imposing
  homogeneous Neumann boundary conditions on the local problem, which is crucial
  for capturing the low-frequency behavior without being constrained by
  artificial boundary effects. The matrix $\mathsf{S}_{i}$ is then defined as a diagonal
  matrix whose entries are the diagonal elements of this modified matrix $\widehat
  {\mathsf{A}_i}$. It serves as a weighting or mass matrix, scaling the eigenvalue
  problem appropriately. This leads to the algebraic generalized eigenvalue
  problem:
  \begin{equation}
    \label{eq:algebraic-eigenproblem}\widehat{\mathsf{A}}_{i}\boldsymbol{\psi}_{k}
    ^{(i)}= \lambda_{k}^{(i)}\mathsf{S}_{i}\boldsymbol{\psi}_{k}^{(i)}.
  \end{equation}
  The eigenvectors $\{\boldsymbol{\psi}_{k}^{(i)}\}$ corresponding to the smallest
  eigenvalues of \cref{eq:algebraic-eigenproblem} represent the local low-energy
  modes. These are precisely the modes that are poorly attenuated by classical
  relaxation methods. By incorporating them into our coarse space, we can
  effectively address these problematic components on a global scale. We select the
  first $m_{i}$ eigenvectors from each subdomain $\Omega_{i}$ to form the basis for
  our coarse space. These local vectors are then extended by zero to the global
  fine-grid space and assembled as the columns of the prolongation operator
  $\mathsf{R}_{0}^{\intercal}$:
  \begin{equation}
    \mathsf{R}_{0}^{\intercal}= [\boldsymbol{\psi}_{1}^{(1)}, \dots, \boldsymbol{\psi}
    _{m_1}^{(1)}, \boldsymbol{\psi}_{1}^{(2)}, \dots, \boldsymbol{\psi}_{m_2}^{(2)}
    , \dots, \boldsymbol{\psi}_{1}^{(N_c)}, \dots, \boldsymbol{\psi}_{m_{N_c}}^{(N_c)}
    ].
  \end{equation}
  This construction ensures that the coarse space is inherently adapted to the properties
  of the underlying operator, making it highly effective for resolving
  challenging solution components. The purely algebraic nature of this process renders
  our method broadly applicable, particularly for problems defined on
  unstructured meshes or in scenarios where only the system matrix is accessible.

  % \begin{algorithm}[!ht]
  %   \caption{Two-grid preconditioner for solving $\mathsf{A}\mathsf{x} = \mathsf{f}$}
  %   \label{alg:two-grid}
  %   \begin{algorithmic}[1]
  %     \State Given initial guess $\mathsf{x}^0$, \textbf{do}
  %     \State $\mathsf{x}^1 = \mathsf{x}^0 + \mathsf{S}(\mathsf{f} - \mathsf{A}\mathsf{x}^0)$ \Comment{Pre-smoothing, $\mathsf{S}$ is a smoother}
  %     \State $\mathsf{r}_1 = \mathsf{R}(\mathsf{f} - \mathsf{A}\mathsf{x}^1)$ \Comment{Compute coarse-grid residual}
  %     \State Solve $(\mathsf{R}\mathsf{A}\mathsf{R}^T)\mathsf{x}_c = \mathsf{r}_1$ \Comment{coarse-grid correction}
  %     \State $\mathsf{x}^2 = \mathsf{x}^1 + \mathsf{R}^T\mathsf{x}_c$ \Comment{Project coarse-grid correction to fine grid}
  %     \State $\mathsf{x}^3 = \mathsf{x}^2 + \mathsf{S}(\mathsf{f} - \mathsf{A}\mathsf{x}^2)$ \Comment{Post-smoothing, $\mathsf{S}$ is a smoother}
  %   \end{algorithmic}
  % \end{algorithm}

  \section{Analysis}
  \label{sec:anal}

  For iterative solvers, a smaller iteration number usually means a better performance.
  In this section, we will analyze the convergence rate of the proposed two-level
  method following \cite{xu2017algebraic}.

  \begin{theorem}
    \label{thm:condition} Suppose $\mathsf{A}$ and $\mathsf{P}$ be symmetric positive
    definite matrices. Let $\mathsf{E}= \mathsf{I}- \mathsf{P}^{-1}\mathsf{A}$. Then,
    the estimation of $\|\mathsf{E}\|_{\mathsf{A}}$ is equivalent to the
    estimation of the condition number $\kappa(\mathsf{P}^{-1}\mathsf{A})$. Specifically,
    if $\|\mathsf{E}\|_{\mathsf{A}}\le \delta < 1$, then
    \[
      \kappa(\mathsf{P}^{-1}\mathsf{A}) \le \frac{1+\delta}{1-\delta}.
    \]
  \end{theorem}

  \begin{proof}
    Let $\mathsf{M}= \mathsf{P}^{-1}\mathsf{A}$. Since $\mathsf{A}$ and
    $\mathsf{P}$ are symmetric positive definite, $\mathsf{M}$ is similar to the
    symmetric positive definite matrix $\widetilde{\mathsf{M}}= \mathsf{A}^{1/2}\mathsf{P}
    ^{-1}\mathsf{A}^{1/2}$. Thus, $\mathsf{M}$ has real, positive eigenvalues
    denoted by $\lambda_{i}$. We first show that $\mathsf{E}$ is self-adjoint with
    respect to the $\mathsf{A}$-inner product, defined as
    $\langle \mathsf{x}, \mathsf{y}\rangle_{\mathsf{A}}= \mathsf{x}^{T}\mathsf{A}
    \mathsf{y}$. The self-adjoint property requires $\langle \mathsf{E}\mathsf{x}
    , \mathsf{y}\rangle_{\mathsf{A}}= \langle \mathsf{x}, \mathsf{E}\mathsf{y}\rangle
    _{\mathsf{A}}$ for all $\mathsf{x}, \mathsf{y}$. This is equivalent to
    \[
      \mathsf{x}^{T}\mathsf{E}^{T}\mathsf{A}\mathsf{y}= \mathsf{x}^{T}\mathsf{A}\mathsf{E}
      \mathsf{y}\iff \mathsf{E}^{T}\mathsf{A}= \mathsf{A}\mathsf{E}.
    \]
    Since $\mathsf{A}$ is symmetric, the condition becomes that $\mathsf{A}\mathsf{E}$
    must be symmetric. Substituting
    $\mathsf{E}= \mathsf{I}- \mathsf{P}^{-1}\mathsf{A}$, we have
    \[
      \mathsf{A}\mathsf{E}= \mathsf{A}(\mathsf{I}- \mathsf{P}^{-1}\mathsf{A})= \mathsf{A}
      - \mathsf{A}\mathsf{P}^{-1}\mathsf{A}.
    \]
    Since $\mathsf{A}$ and $\mathsf{P}$ (and thus $\mathsf{P}^{-1}$) are symmetric,
    $\mathsf{A}\mathsf{P}^{-1}\mathsf{A}$ is symmetric, which implies $\mathsf{A}
    \mathsf{E}$ is symmetric. Therefore, $\mathsf{E}$ is self-adjoint in the $\mathsf{A}$-inner
    product, and
    $\|\mathsf{E}\|_{\mathsf{A}}= \rho(\mathsf{E})= \max_{i}|1 - \lambda_{i}|$. If
    $\|\mathsf{E}\|_{\mathsf{A}}\le \delta$, then $|1 - \lambda_{i}| \le \delta$
    for all $i$, which implies $1 - \delta \le \lambda_{\min}\le \lambda_{\max}\le
    1 + \delta$. Consequently,
    \[
      \kappa(\mathsf{P}^{-1}\mathsf{A}) = \frac{\lambda_{\max}}{\lambda_{\min}}\le
      \frac{1+\delta}{1-\delta}.
    \]
  \end{proof}
  Let $\mathsf{S}$ be the block diagonal matrix induced by the local matrices
  $\{ \mathsf{S}_{i}\}_{i=1}^{N}$ on the non-overlapping partition
  $\{\Omega_{i}\}_{i=1}^{N}$. We define the corresponding weighted norms by
  \[
    \|x\|_{\mathsf{S}}^{2}= x^{T}\mathsf{S}x, \qquad \|x_{i}\|_{\mathsf{S}_{i}}^{2}
    = x_{i}^{T}\mathsf{S}_{i}x_{i}.
  \]
  \begin{lemma}
    For any set of local vectors $\{x_{i}\}_{i=1}^{N}$ where $x_{i}\in \mathbb{R}
    ^{n_{i}}$, the following equality holds:

    \[
      \left\|\sum_{i=1}^{N}\mathsf{R}_{i}^{\intercal}x_{i}\right\|_{\mathsf{S}}^{2}
      = \sum_{i=1}^{N}\|x_{i}\|_{\mathsf{S}_{i}}^{2}
    \]
  \end{lemma}

  \begin{proof}
    Let $v = \sum_{i=1}^{N}\mathsf{R}_{i}^{\intercal}x_{i}$. By definition of
    the $\mathsf{S}$-norm,
    \[
      \|v\|_{\mathsf{S}}^{2}= v^{T}\mathsf{S}v.
    \]
    Since the subdomains are non-overlapping and $\mathsf{R}_{i}^{\intercal}$ extends
    a local vector by zero outside $\Omega_{i}$, the supports of
    $\mathsf{R}_{i}^{\intercal}x_{i}$ are disjoint. Hence there are no cross terms
    in the quadratic form associated with the block diagonal matrix $\mathsf{S}$,
    and
    \[
      \|v\|_{\mathsf{S}}^{2}= \sum_{i=1}^{N}(\mathsf{R}_{i}^{\intercal}x_{i})^{T}
      \mathsf{S}(\mathsf{R}_{i}^{\intercal}x_{i}) = \sum_{i=1}^{N}x_{i}^{T}\mathsf{S}
      _{i}x_{i}= \sum_{i=1}^{N}\|x_{i}\|_{\mathsf{S}_{i}}^{2}.
    \]
  \end{proof}

  \begin{lemma}
    For each $x\in \mathbb{R}^{n}$, there exists vectors
    $x_{1}, x_{2}, \cdots, x_{N}$ such that
    $x = \sum_{i=1}^{N}\mathsf{R}_{i}^{\intercal}x_{i}$ and
    \[
      \sum_{i=1}^{N}\|x_{i}\|_{\widehat{\mathsf{A}}_{i}}^{2}\leq \|x\|_{\mathsf{A}}
      ^{2}.
    \]
  \end{lemma}

  \begin{proof}
    For each subdomain $\Omega_{i}$, let $x_{i}= \mathsf{R}_{i}x$ be the
    restriction of $x$ to $\Omega_{i}$. Then we trivially have $x = \sum_{i=1}^{N}
    \mathsf{R}_{i}^{\intercal}x_{i}$.

    Recall that $\widehat{\mathsf{A}}_{i}$ is constructed by modifying the diagonal
    entries of the local matrix such that its row sums are zero. For symmetric
    matrices with non-positive off-diagonal entries (M-matrices), the associated
    quadratic form can be written as a sum over edges:
    \[
      \|x_{i}\|_{\widehat{\mathsf{A}}_{i}}^{2}= \frac{1}{2}\sum_{\substack{j,k \in \Omega_i \\ j \neq k}}
      (-\mathsf{A}_{jk}) (x_{j}- x_{k})^{2}.
    \]
    Summing over all subdomains gives the total energy associated with all
    interior edges:
    \[
      \sum_{i=1}^{N}\|x_{i}\|_{\widehat{\mathsf{A}}_{i}}^{2}= \sum_{i=1}^{N}\sum_{\substack{j,k \in \Omega_i \\ j < k}}
      (-\mathsf{A}_{jk}) (x_{j}- x_{k})^{2}.
    \]
    On the other hand, the global energy norm $\|x\|_{\mathsf{A}}^{2}$ includes
    contributions from all edges in the graph as well as diagonal terms:
    \[
      \|x\|_{\mathsf{A}}^{2}= \sum_{j=1}^{n}\left(\sum_{k=1}^{n}\mathsf{A}_{jk}\right
      )x_{j}^{2}+ \sum_{j < k}(-\mathsf{A}_{jk})(x_{j}- x_{k})^{2}.
    \]
    Since $\mathsf{A}$ is a stiffness matrix derived from an elliptic problem,
    we have $-\mathsf{A}_{jk}\ge 0$ for $j \neq k$ and the row sums
    $\sum_{k}\mathsf{A}_{jk}\ge 0$ (by weak diagonal dominance). Decomposing the
    edge sum into interior edges (where both nodes are in the same $\Omega_{i}$)
    and cut edges (where nodes belong to different subdomains), we obtain:
    \[
      \sum_{j < k}(-\mathsf{A}_{jk})(x_{j}- x_{k})^{2}= \sum_{i=1}^{N}\sum_{\substack{j,k \in \Omega_i \\ j < k}}
      (-\mathsf{A}_{jk}) (x_{j}- x_{k})^{2}+ \sum_{\text{cut edges}}(-\mathsf{A}_{jk}
      )(x_{j}- x_{k})^{2}.
    \]
    Comparing the expressions, we see that
    $\sum_{i=1}^{N}\|x_{i}\|_{\widehat{\mathsf{A}}_{i}}^{2}$ corresponds exactly
    to the first term of the global energy decomposition. Since the contributions
    from cut edges and row sums are non-negative, it follows immediately that:
    \[
      \sum_{i=1}^{N}\|x_{i}\|_{\widehat{\mathsf{A}}_{i}}^{2}\le \|x\|_{\mathsf{A}}
      ^{2}.
    \]
  \end{proof}
  Let $E$ denote the error propagation operator of the two-level method. Following
  the abstract convergence theorem \cite{xu2017algebraic}, we have
  \begin{theorem}
    \label{thm:convergence} Let
    \[
      \mu_{c}\coloneqq \min_{1 \le i \le N}\lambda_{m_i+1}^{(i)},
    \]
    where $\lambda_{m_i+1}^{(i)}$ is the first local eigenvalue not included in the
    coarse space on $\Omega_{i}$. Then the error propagation operator $E$ of the
    two-level method satisfies
    \[
      \|E\|_{\mathsf{A}}^{2}\le 1 - \frac{\mu_{c}}{C},
    \]
    where $C$ is a fixed constant.
  \end{theorem}

  \begin{proof}
    For each subdomain $\Omega_{i}$, let $\Pi_{i}$ be the $\mathsf{S}_{i}$-orthogonal
    projection onto the span of the selected eigenvectors
    $\{\boldsymbol{\psi}_{1}^{(i)}, \dots, \boldsymbol{\psi}_{m_i}^{(i)}\}$. For
    any $x \in \mathbb{R}^{n}$, choose the decomposition
    $x = \sum_{i=1}^{N}\mathsf{R}_{i}^{\intercal}x_{i}$ given by the previous
    lemma, and define the coarse component by
    \[
      x_{c}\coloneqq \sum_{i=1}^{N}\mathsf{R}_{i}^{\intercal}\Pi_{i}x_{i}\in W_{H}
      ^{\text{c}}.
    \]

    By the spectral characterization of \cref{eq:algebraic-eigenproblem}, the part
    orthogonal to the selected local eigenspace satisfies
    \[
      \|x_{i}- \Pi_{i}x_{i}\|_{\mathsf{S}_{i}}^{2}\le \frac{1}{\lambda_{m_i+1}^{(i)}}
      \|x_{i}\|_{\widehat{\mathsf{A}}_{i}}^{2}\le \frac{1}{\mu_{c}}\|x_{i}\|_{\widehat{\mathsf{A}}_{i}}
      ^{2}.
    \]
    Therefore, using the two previous lemmas,
    \begin{align*}
      \|x - x_{c}\|_{\mathsf{S}}^{2} & = \left\|\sum_{i=1}^{N}\mathsf{R}_{i}^{\intercal}(x_{i}- \Pi_{i}x_{i})\right\|_{\mathsf{S}}^{2} \\
                                     & = \sum_{i=1}^{N}\|x_{i}- \Pi_{i}x_{i}\|_{\mathsf{S}_{i}}^{2}                                    \\
                                     & \le \frac{1}{\mu_{c}}\sum_{i=1}^{N}\|x_{i}\|_{\widehat{\mathsf{A}}_{i}}^{2}                     \\
                                     & \le \frac{1}{\mu_{c}}\|x\|_{\mathsf{A}}^{2}.
    \end{align*}

    Following \cite{xu2017algebraic}, the convergence rate of the two-level method
    is
    \[
      \|E\|_{\mathsf{A}}^{2}= 1 - \frac{1}{CK},
    \]
    where $K = \max_{x \in \mathbb{R}^n}\min_{x_{c}\in W_{H}^{\text{c}}}\frac{\|x - x_{c}\|_{\mathsf{S}}^{2}}{\|x\|_{\mathsf{A}}^{2}}$ and $C$ is a constant depending on the largest eigenvalue of the symmetrized smoother.

    % Hence the approximation and stability assumptions of the abstract theory are
    % satisfied with constant
    By the above inequality $K = 1/\mu_{c}$. Combining the above results we have
    \[
      \|E\|_{\mathsf{A}}^{2}= 1 - \frac{1}{CK}\le 1 - \frac{\mu_{c}}{C}.
    \]
    This completes the proof.
  \end{proof}

  Combining \Cref{thm:condition} and \Cref{thm:convergence}, we know that
  the condition number of the preconditioned system is bounded above by a constant.

  \section{Numerical Experiments}
  \label{sec:experiments} In this section, we present a series of numerical
  experiments to demonstrate the effectiveness and robustness of the proposed
  algebraic multiscale preconditioner. For the numerical tests, we consider the
  steady-state heat conduction problem within the unit cube $\Omega = [0,1]^{3}$,
  governed by the following partial differential equation with homogeneous Neumann
  boundary conditions:
  \begin{equation}
    \left\{
    \begin{aligned}
      -\text{div}(\kappa\nabla T)     & =f\quad  &  & \text{in} \quad \Omega,        \\
      \kappa\nabla T \cdot \mathsf{n} & = 0\quad &  & \text{on}\quad \partial\Omega,
    \end{aligned}
    \right.
  \end{equation}
  To discretize the model problem, we employ the lowest-order Raviart-Thomas ($RT
  _{0}$) mixed finite element method \cite{boffi2013mixed}. The resulting saddle-point
  system is then reduced to a symmetric positive definite system for the
  pressure variables via a velocity elimination technique
  \cite{arbogast1997mixed,arbogast1997mixed,russell1983finite,chen2020generalized}.
  Technically, we use the MFEM library \cite{anderson2021mfem} to handle the
  mesh and finite element discretization.

  Our implementation is fully parallelized using the Message Passing Interface (MPI).
  The graph associated with the system matrix is first partitioned among the MPI
  processes. Each MPI process owns a local portion of the graph together with
  the corresponding rows of the distributed linear system. We denote the total number
  of processes by $\mathsf{proc}$. For the construction of the multiscale coarse
  space, each process further partitions its local subgraph into $\mathsf{sd}$ smaller
  subdomains using the METIS graph partitioning library \cite{karypis1997metis, karypis1997parmetis}.
  Consequently, the global graph is decomposed into a total of $N=\mathsf{proc}\times
  \mathsf{sd}$ non-overlapping subdomains, which define the coarse partition
  used by our method. The partitioning is designed to balance the computational workload
  across processes while keeping the number of cut edges small, thereby reducing
  communication overhead and improving parallel efficiency.

  The local generalized eigenvalue problems \cref{eq:algebraic-eigenproblem} on
  each subdomain are solved using SLEPc \cite{hernandez2005slepc} to construct
  the multiscale basis functions. More specifically, to accelerate the local
  spectral computations, each generalized eigenvalue problem is transformed into
  an equivalent standard symmetric eigenvalue problem through a Cholesky factorization
  of the weighting matrix. From each subdomain, we retain the first $L_{*}$ eigenvectors
  associated with the smallest eigenvalues, since these modes capture the
  dominant local low-energy features that must be represented in the coarse
  space. The resulting coarse-scale system is then solved directly by the parallel
  sparse direct solver $\text{SUPERLU\_DIST}$ \cite{li2003superlu_dist}. On the
  fine scale, we employ the generalized minimal residual method (GMRES) preconditioned
  by the proposed two-level method. The iteration is terminated once the
  relative residual has been reduced by a factor of $10^{-6}$. Throughout the numerical
  experiments, we report the iteration count as $\mathtt{iter}$ and denote the
  total number of degrees of freedom by DoF.

  All computations are performed on a high-performance computing (HPC) cluster. Each
  node in the cluster is equipped with dual Intel\textsuperscript{\textregistered}
  Xeon\textsuperscript{\textregistered} Gold 6258R CPUs (totaling 56 cores) and 192GB
  of memory, interconnected by an Infiniband network. The source code for our implementation
  is publicly available on GitHub\footnote{\href{https://github.com/pentaery/Algebraic-2G}{https://github.com/pentaery/Algebraic-2G}}.

  \subsection{Contrast robustness}
  In this subsection, we test our preconditioner on media containing long channels,
  as shown in \Cref{fig:permeability1}. We fix the coefficient $\kappa = 1$ in the
  background region and increase the coefficient in the high-conductivity structures
  to examine robustness with respect to contrast ratios $10^{\mathsf{cr}}$. We
  consider three problem sizes with $\mathtt{DoF}=128^{3}$, $256^{3}$, and
  $512^{3}$.

  The results are summarized in
  \Cref{tab:comparison_contrast_128,tab:comparison_contrast_256,tab:comparison_contrast_512}.
  Several clear trends can be read from these tables. First, for the proposed method,
  the setup cost is only weakly affected by the contrast ratio because the
  number of local eigenpairs and the coarse-space construction parameters are
  kept fixed. This is particularly visible for $\mathtt{DoF}=128^{3}$, where the
  setup time remains $0.3$ seconds for all three contrasts, and for $\mathtt{DoF}
  =256^{3}$, where it varies only between $0.7$ and $0.8$ seconds. More importantly,
  the iteration counts remain almost unchanged on these two grids: they vary only
  from $43$ to $45$ for $128^{3}$ and from $51$ to $54$ for $256^{3}$ as the contrast
  increases from $10^{3}$ to $10^{5}$. As a result, the total solution time stays
  in a narrow range, namely $0.4$--$0.7$ seconds for $128^{3}$ and $1.3$--$1.8$
  seconds for $256^{3}$. This indicates that, for these problem sizes, the local
  spectral coarse space already captures the dominant low-energy modes induced by
  the high-conductivity channels.

  The behavior of PCGAMG is markedly different. Its setup time changes only
  moderately with contrast, but its solve time and iteration count deteriorate rapidly
  as the contrast increases, showing that the main difficulty lies in the
  quality of the coarse correction rather than in the setup stage. For $\mathtt{DoF}
  =128^{3}$, the iteration count increases from $63$ to $293$ and the total time
  rises from $0.6$ to $1.4$ seconds. For $\mathtt{DoF}=256^{3}$, the
  deterioration is much stronger: the iteration count grows from $78$ to $474$,
  the solve time increases from $1.0$ to $5.7$ seconds, and the total time rises
  from $1.7$ to $6.2$ seconds. In contrast, our method keeps the total time at $1
  .3$, $1.8$, and $1.6$ seconds for the same three contrasts, corresponding to a
  reduction of about $24\%$, $54\%$, and $74\%$, respectively. These data show
  that the proposed multiscale preconditioner is substantially less sensitive to
  coefficient contrast than standard AMG.

  The largest problem size, $\mathtt{DoF}=512^{3}$, reveals a more nuanced but
  also more informative picture. For contrasts $10^{3}$ and $10^{4}$, our method
  still achieves very small iteration counts, namely $31$ and $32$, compared
  with $60$ and $203$ for PCGAMG. However, at the lower contrast $10^{3}$, the cheaper
  AMG setup makes PCGAMG slightly faster overall ($10.4$ seconds versus $11.5$
  seconds), even though our method requires about half as many iterations. This shows
  that, in relatively easier regimes, the additional cost of the local eigenvalue
  solves may not yet be amortized. When the contrast is increased to $10^{4}$, the
  advantage of the proposed method becomes clear: the total time drops from
  $28.3$ seconds for PCGAMG to $15.6$ seconds for our method. At the extreme contrast
  $10^{5}$, the iteration count of our method increases to $160$, which indicates
  some loss of contrast-independence on the largest grid under this fixed parameter
  choice. Nevertheless, it still significantly outperforms PCGAMG, which
  requires $470$ iterations and $61.3$ seconds in total, whereas our method needs
  $40.7$ seconds. Therefore, the results suggest that the proposed method is robust
  over a broad range of contrasts and problem sizes, and that even in the most
  difficult tested regime it remains substantially more effective than standard
  AMG, although a richer coarse space may be beneficial for the most extreme case.

  \begin{figure*}[!ht]
    \centering
    \begin{subfigure}
      [b]{0.45\textwidth}
      \includegraphics[width=\textwidth]{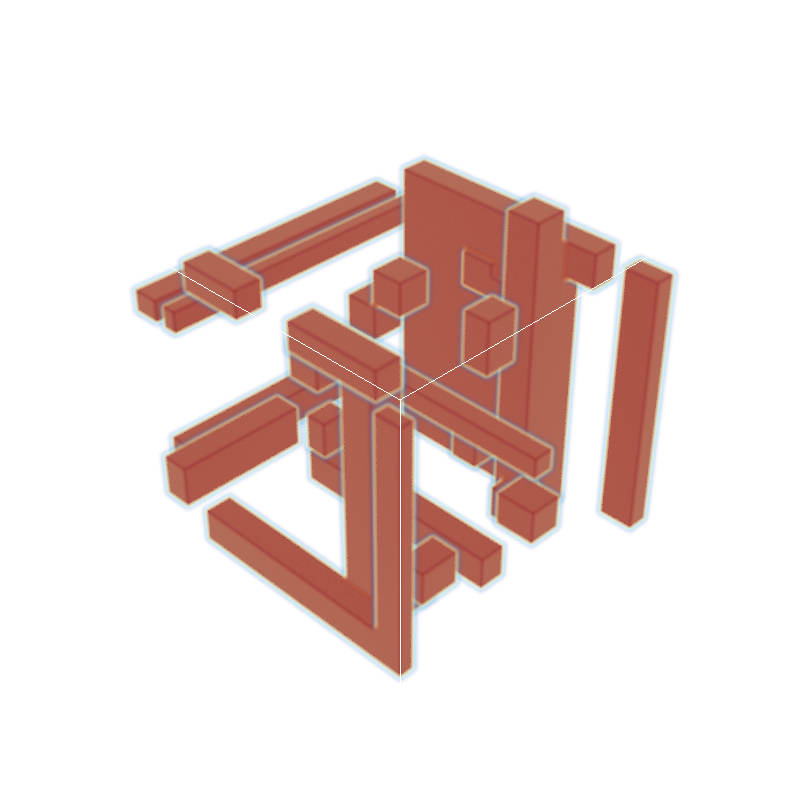}
      \caption{}
    \end{subfigure}
    \hfill
    \begin{subfigure}
      [b]{0.45\textwidth}
      \includegraphics[width=\textwidth]{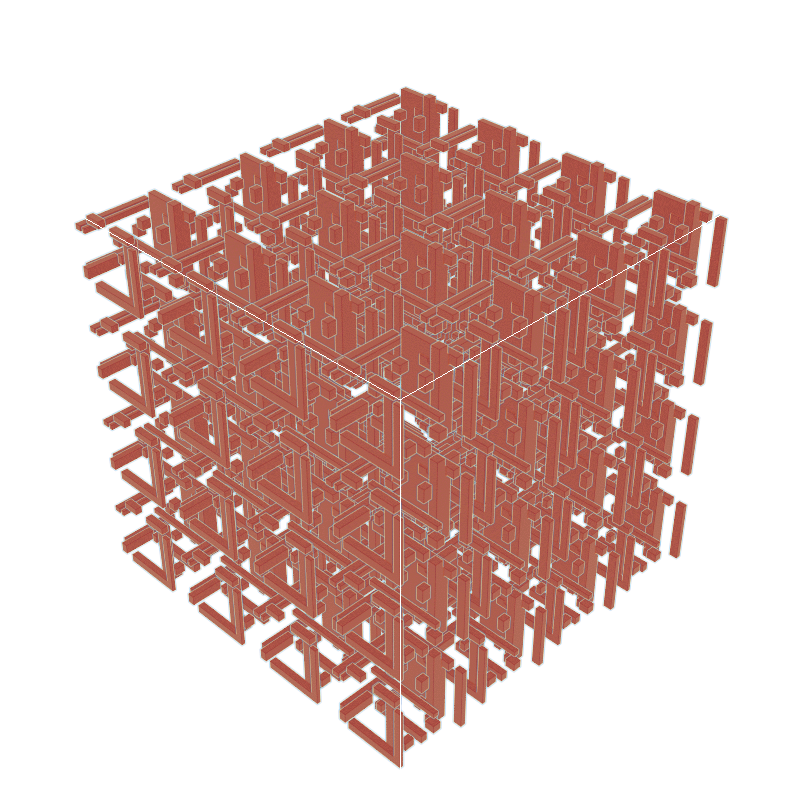}
      \caption{}
    \end{subfigure}
    \caption{First kind of permeability field with long channels.}
    \label{fig:permeability1}
  \end{figure*}

  \begin{figure*}[!ht]
    \centering
    \begin{subfigure}
      [b]{0.45\textwidth}
      \includegraphics[width=\textwidth]{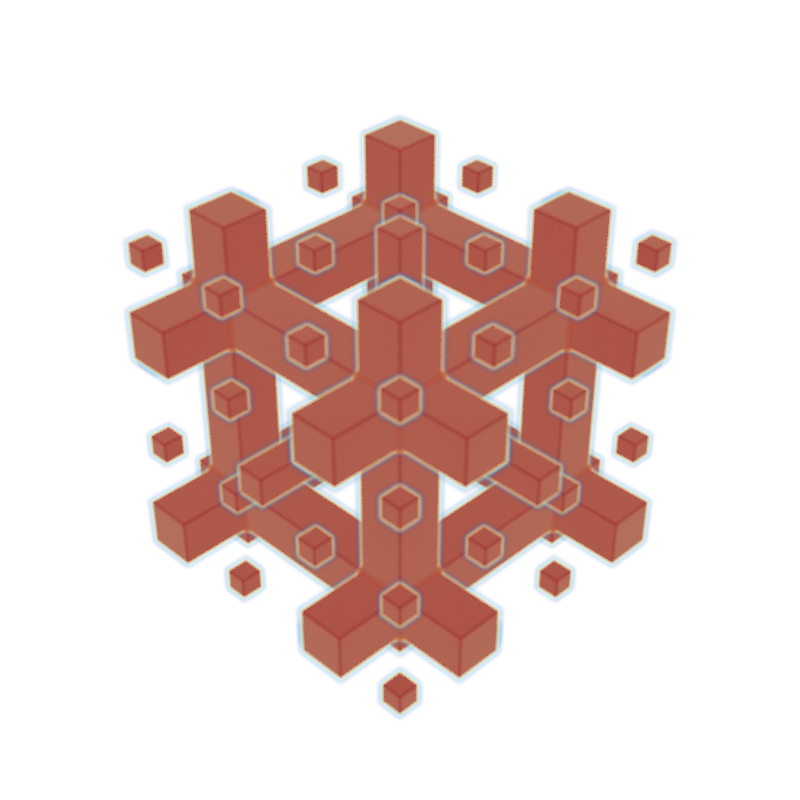}
      \caption{}
    \end{subfigure}
    \hfill
    \begin{subfigure}
      [b]{0.45\textwidth}
      \includegraphics[width=\textwidth]{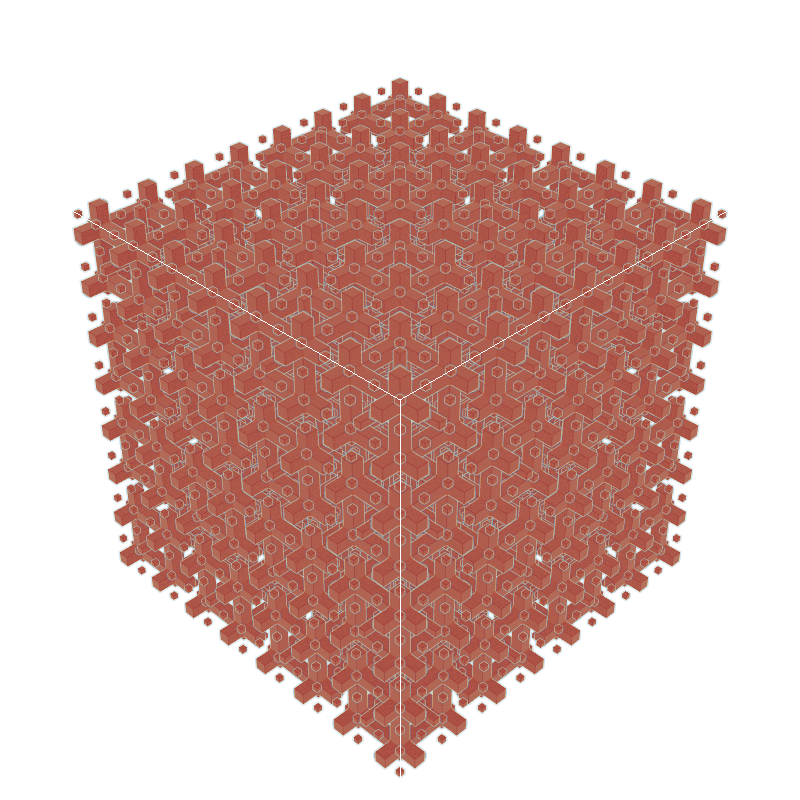}
      \caption{}
    \end{subfigure}
    \caption{Second kind of permeability field with fractures.}
    \label{fig:permeability2}
  \end{figure*}

  \begin{table*}
    [!ht] \footnotesize
    \caption{Records of elapsed wall time and iteration numbers of AMG and the
    proposed preconditioner w.r.t.\ different contrast ratios ($\mathtt{cr}$).
    The parameters for the proposed method are fixed, and $\mathtt{DoF}=128^{3}$,
    we use 560 MPI processes.}
    \label{tab:comparison_contrast_128}
    \begin{center}
      \makegapedcells
      \begin{tabular}{c c c c c c}
        \toprule Contrast                  & Preconditioner & Setup (s) & Solve (s) & Iterations & Total (s) \\
        \midrule \multirow{2}{*}{$10^{3}$} & PCGAMG         & $0.1$     & $0.5$     & $63$       & $0.6$     \\
                                           & Ours           & $0.3$     & $0.1$     & $43$       & $0.4$     \\
        \midrule \multirow{2}{*}{$10^{4}$} & PCGAMG         & $0.5$     & $0.7$     & $108$      & $1.2$     \\
                                           & Ours           & $0.3$     & $0.4$     & $44$       & $0.7$     \\
        \midrule \multirow{2}{*}{$10^{5}$} & PCGAMG         & $0.5$     & $0.9$     & $293$      & $1.4$     \\
                                           & Ours           & $0.3$     & $0.4$     & $45$       & $0.7$     \\
        \bottomrule
      \end{tabular}
    \end{center}
  \end{table*}

  \begin{table*}
    [!ht] \footnotesize
    \caption{Records of elapsed wall time and iteration numbers of AMG and the
    proposed preconditioner w.r.t.\ different contrast ratios ($\mathtt{cr}$).
    The parameters for the proposed method are fixed, and $\mathtt{DoF}=256^{3}$,
    we use 560 MPI processes.}
    \label{tab:comparison_contrast_256}
    \begin{center}
      \makegapedcells
      \begin{tabular}{c c c c c c}
        \toprule Contrast                  & Preconditioner & Setup (s) & Solve (s) & Iterations & Total (s) \\
        \midrule \multirow{2}{*}{$10^{3}$} & PCGAMG         & $0.7$     & $1.0$     & $78$       & $1.7$     \\
                                           & Ours           & $0.7$     & $0.6$     & $51$       & $1.3$     \\
        \midrule \multirow{2}{*}{$10^{4}$} & PCGAMG         & $0.8$     & $3.1$     & $235$      & $3.9$     \\
                                           & Ours           & $0.8$     & $1.0$     & $53$       & $1.8$     \\
        \midrule \multirow{2}{*}{$10^{5}$} & PCGAMG         & $0.5$     & $5.7$     & $474$      & $6.2$     \\
                                           & Ours           & $0.7$     & $0.9$     & $54$       & $1.6$     \\
        \bottomrule
      \end{tabular}
    \end{center}
  \end{table*}

  \begin{table*}
    [!ht] \footnotesize
    \caption{Records of elapsed wall time and iteration numbers of AMG and the
    proposed preconditioner w.r.t.\ different contrast ratios ($\mathtt{cr}$).
    The parameters for the proposed method are fixed, and $\mathtt{DoF}=512^{3}$,
    we use 560 MPI processes.}
    \label{tab:comparison_contrast_512}
    \begin{center}
      \makegapedcells
      \begin{tabular}{c c c c c c}
        \toprule Contrast                  & Preconditioner & Setup (s) & Solve (s) & Iterations & Total (s) \\
        \midrule \multirow{2}{*}{$10^{3}$} & PCGAMG         & $2.9$     & $7.5$     & $60$       & $10.4$    \\
                                           & Ours           & $6.5$     & $5.0$     & $31$       & $11.5$    \\
        \midrule \multirow{2}{*}{$10^{4}$} & PCGAMG         & $4.3$     & $24.0$    & $203$      & $28.3$    \\
                                           & Ours           & $10.0$    & $5.6$     & $32$       & $15.6$    \\
        \midrule \multirow{2}{*}{$10^{5}$} & PCGAMG         & $3.1$     & $58.2$    & $470$      & $61.3$    \\
                                           & Ours           & $10.4$    & $29.3$    & $160$      & $40.7$    \\
        \bottomrule
      \end{tabular}
    \end{center}
  \end{table*}

  \subsection{Scalability test}
  \begin{figure*}[!ht]
    \centering
    \resizebox{\textwidth}{!}{\input{Scaling.pgf}}
    \caption{Strong and weak scalability results for the proposed preconditioner.
    The problem size is fixed at $512^{3}$ for the strong scalability test,
    while the number of degrees of freedom per processor is fixed at
    approximately $64^{3}$ for the weak scalability test.}
    \label{fig:scalability}
  \end{figure*}

  In this subsection, we investigate the parallel scalability of our proposed
  algebraic multiscale preconditioner. The coefficient field we use are shown in
  \Cref{fig:permeability2}. We conduct both strong and weak scalability tests, and
  the results are presented in \Cref{fig:scalability}. For these tests, we use the
  permeability field with fractures shown in \Cref{fig:permeability2} and set the
  contrast ratio to $10^{4}$.

  For the strong scalability test, we fix the total problem size with
  $\mathtt{DoF}= 512^{3}$ and increase the number of MPI processes from $27$ ($3^{3}$)
  to $512$ ($8^{3}$). As shown in the left panel of \Cref{fig:scalability}, the
  total elapsed time, which includes both setup and solve phases, decreases
  significantly as more processors are employed. The setup time (dark blue bars)
  remains relatively small and scales well, while the solve time (light blue bars)
  shows a substantial reduction. The number of GMRES iterations remains remarkably
  stable, indicating that the preconditioner's effectiveness is maintained as
  the number of subdomains increases. This demonstrates the excellent strong scalability
  of our method.

  For the weak scalability test, we keep the local problem size per processor
  approximately constant at $64^{3}$ DoF, while increasing the total problem size
  and the number of processors proportionally. The number of processors is
  varied from $1 25$ ($5^{3}$) to $512$ ($8^{3}$), with the corresponding total DoF
  growing from $320^{3}$ to $512^{3}$. The right panel of \Cref{fig:scalability}
  shows that the total elapsed time remains nearly constant as the problem size
  and processor count grow. Both the setup and solve times exhibit good weak scalability.
  The number of iterations shows only a very slight increase, confirming that the
  preconditioner's performance does not degrade as the problem scales up. These
  results affirm the robustness and efficiency of our parallel algebraic
  multiscale preconditioner for large-scale simulations.

  \subsection{Parameter test}

  An important advantage of choosing $\mathsf{S}_{i}$ to be diagonal in \Cref{eq:algebraic-eigenproblem}
  is that the generalized eigenvalue problem can be transformed explicitly into a
  standard symmetric eigenvalue problem, namely,
  \[
    \mathsf{S}_{i}^{-1/2}\widehat{\mathsf{A}}_{i}\mathsf{S}_{i}^{-1/2}\mathsf{z}=
    \lambda \mathsf{z}, \quad \text{with}\quad \boldsymbol{\psi}=\mathsf{S}_{i}^{-1/2}
    \mathsf{z},
  \]
  which is computationally more efficient.
  \begin{table*}
    [!ht]
    \begin{center}
      \caption{Elapsed wall time (in seconds) for solving the local eigenvalue
      problems using the generalized eigenvalue problem solver (GHEP) and the
      standard eigenvalue problem solver (HEP). The tests are conducted on multiple
      processors and the times are averaged.}
      \label{tab:GHEPvsHEP}
      \begin{tabular}{cccc}
        \toprule size  & $20^{3}$ & $30^{3}$ & $40^{3}$ \\
        \midrule $HEP$ & $0.1$    & $0.9$    & $3.2$    \\
        $GHEP$         & $0.5$    & $3.4$    & $13.0$   \\
        \bottomrule
      \end{tabular}
    \end{center}
  \end{table*}

  \Cref{tab:GHEPvsHEP} confirms that this reformulation yields a clear practical
  advantage. For all tested local problem sizes, solving the transformed standard
  symmetric eigenvalue problem is about four to five times faster than solving the
  original generalized problem directly. For example, at size $40^{3}$, the
  average time is reduced from $13.0$ seconds for GHEP to $3.2$ seconds for HEP.
  The main reason is that, after diagonal scaling, the eigensolver only needs to
  treat a single symmetric operator, so it can use more efficient algorithms. In
  contrast, the generalized formulation must work with the matrix pencil
  $(\widehat{\mathsf{A}}_{i},\mathsf{S}_{i})$, which introduces additional overhead
  in normalization and projection steps even when $\mathsf{S}_{i}$ is diagonal.
  Since applying $\mathsf{S}_{i}^{-1/2}$ is inexpensive, converting the problem to
  HEP removes this extra generalized-eigensolver cost without changing the spectrum,
  which explains the consistently better timings reported in
  \Cref{tab:GHEPvsHEP}.

  The main computational burden of our preconditioner lies in the balance
  between the eigenvalue solver and the coarse direct solver. The former is responsible
  for constructing a robust coarse space, while the latter ensures that the
  coarse problem can be solved efficiently. In this subsection, we will test the
  performance of our preconditioner with different parameters, including the number
  of subdomains per processor $\mathtt{sd}$, the number of selected eigenvectors
  per subdomain $L_{*}$ and the number of subdomains $\mathsf{sd}$. We fix $\mathtt{cr}
  =4$ and $\mathtt{DoF}=256^{3}$.

  \begin{figure*}[!ht]
    \centering
    \resizebox{\textwidth}{!}{\input{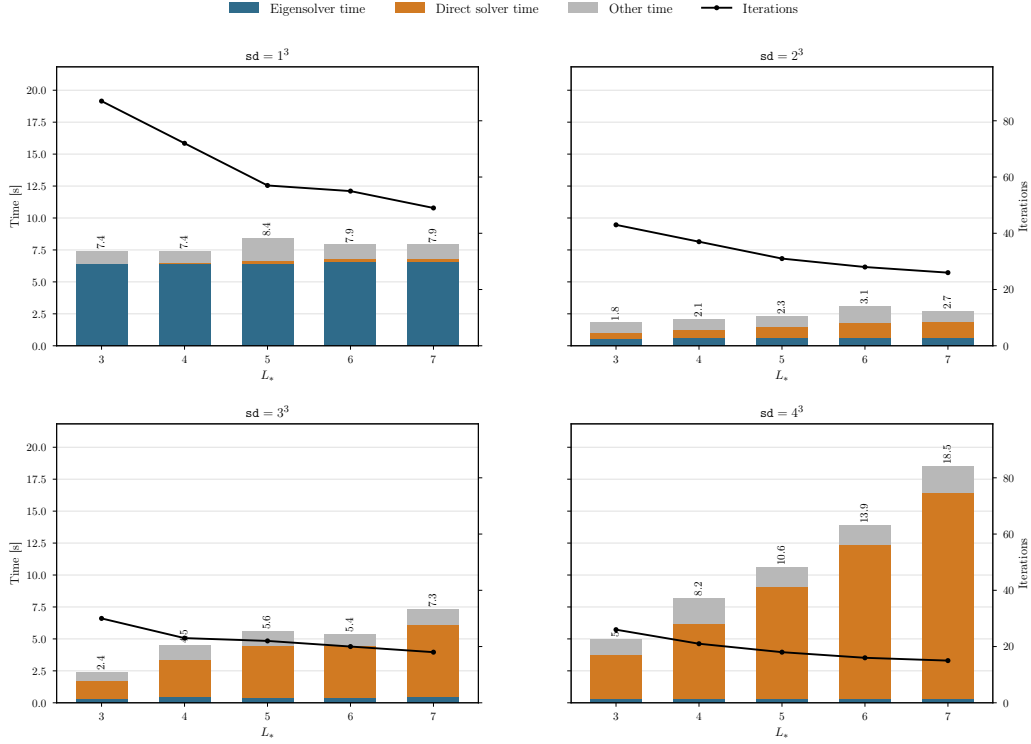}}
    \caption{Strong and weak scalability results for the proposed preconditioner.
    The problem size is fixed at $512^{3}$ for the strong scalability test,
    while the number of degrees of freedom per processor is fixed at
    approximately $64^{3}$ for the weak scalability test.}
    \label{fig:lstar-vs-sd}
  \end{figure*}

  The results in \Cref{fig:lstar-vs-sd} clearly show the trade-off between the
  local eigenvalue solves and the coarse direct solve. For a fixed $\mathtt{sd}$,
  increasing $L_{*}$ always reduces the number of GMRES iterations, since a richer
  coarse space captures more low-energy modes. However, this does not necessarily
  improve the overall runtime. For example, when $\mathtt{sd}=2^{3}$, the
  iteration count decreases from $43$ to $26$ as $L_{*}$ increases from $3$ to
  $7$, but the total time increases from $1.8$ s to $2.7$ s because the cost of the
  coarse problem grows steadily. A similar trend is even more pronounced for
  $\mathtt{sd}=3^{3}$ and $4^{3}$, where the additional reduction in iterations is
  outweighed by the rapidly increasing direct-solver time.

  The number of subdomains per processor also has a non-monotone influence on the
  performance. Moving from $\mathtt{sd}=1^{3}$ to $\mathtt{sd}=2^{3}$ yields a
  substantial speedup: the eigensolver time drops from about $6.4$ s to $0.6$ s,
  while the iteration count is also reduced significantly. In contrast, choosing
  too many subdomains enlarges the global coarse problem and makes the direct
  solve the dominant cost. For instance, with $L_{*}=5$, the total time is only $2
  .3$ s for $\mathtt{sd}=2^{3}$, but it increases to $5.6$ s and $10.6$ s for
  $\mathtt{sd}=3^{3}$ and $4^{3}$, respectively. Overall, these results indicate
  that a moderate partition together with a small number of local eigenvectors provides
  the best balance between robustness and efficiency. In particular, $\mathtt{sd}
  =2^{3}$ with $L_{*}=3$ or $4$ appears to be the most cost-effective choice for
  this test.

  \section{Conclusion}
  \label{sec:conclusion}

  In this paper, we developed a two-level algebraic multiscale preconditioner for
  large sparse symmetric positive definite systems arising from elliptic problems
  with highly heterogeneous coefficients. The proposed method constructs a
  problem-dependent coarse space in a fully algebraic manner by combining graph
  partitioning with local spectral basis functions. As a result, it does not
  require geometric grids or mesh information and is therefore well suited to
  unstructured discretizations and matrix-only settings. In addition, the local
  nature of the construction makes the method naturally parallelizable.

  Numerical experiments for heterogeneous Darcy flow problems show that the
  proposed preconditioner is robust with respect to coefficient contrast and
  problem size. Compared with standard algebraic multigrid, it delivers smaller
  iteration counts and better overall performance for challenging large-scale
  problems, while maintaining good strong and weak scalability. These results
  indicate that the algebraic multiscale coarse space is effective in capturing
  the low-energy error components that hinder the convergence of classical
  iterative solvers.

  Future work will focus on extending the present two-level framework to
  multilevel variants, developing adaptive strategies for selecting local basis functions,
  and applying the method to a broader class of partial differential equations
  and coupled multiphysics problems.



  \bibliographystyle{elsarticle-num}
  \bibliography{refs.bib}
\end{document}